\documentclass[12pt]{amsart}
\usepackage[latin1]{inputenc}
\usepackage{amssymb,enumerate,amsthm,yfonts,mathrsfs}
\usepackage{amsmath}
\usepackage{bbm}           % used for blackboard 1
\usepackage{bm} 
\usepackage{eso-pic,graphicx}
\usepackage{tikz}

\usepackage[colorlinks=true, pdfstartview=FitV, linkcolor=blue, citecolor=blue, urlcolor=blue]{hyperref}

\makeatletter
\def\Ddots{\mathinner{\mkern1mu\raise\p@
\vbox{\kern7\p@\hbox{.}}\mkern2mu
\raise4\p@\hbox{.}\mkern2mu\raise7\p@\hbox{.}\mkern1mu}}
\makeatother

\def\Xint#1{\mathchoice
{\XXint\displaystyle\textstyle{#1}}%
{\XXint\textstyle\scriptstyle{#1}}%
{\XXint\scriptstyle\scriptscriptstyle{#1}}%
{\XXint\scriptscriptstyle\scriptscriptstyle{#1}}%
\!\int}
\def\XXint#1#2#3{{\setbox0=\hbox{$#1{#2#3}{\int}$}
\vcenter{\hbox{$#2#3$}}\kern-.5\wd0}}

\def\dashint{\Xint-}

\newtheorem{theorem}{Theorem}[section]
\newtheorem{nonum}{Theorem}

\newtheorem{conjecture}[theorem]{Conjecture}

\theoremstyle{definition}

\newtheorem{question}[theorem]{Question}
\newtheorem{observation}[theorem]{Observation}
\newtheorem{lemma}[theorem]{Lemma}

\def\ep{{\epsilon}}

\def\al{{\alpha}}
\def\R{\mathbb R}

\def\Z{\mathbb Z}

\def\ra{\rightarrow}
\def\bey{\begin{eqnarray*}}
\def\eey{\end{eqnarray*}}
 
 \allowdisplaybreaks

\def\D{{\mathscr D}}

\def\Q{{\mathcal Q}}
\def\Ca{{\mathcal C}}

\def\Sp{{\mathcal S}}

\begin{document}

\title[Mixed estimates with one supremum]{Mixed $A_p$-$A_\infty$ estimates with one supremum}
\author{Andrei K. Lerner and Kabe Moen }
 \address{Department of Mathematics, Bar-Ilan University, 52900 Ramat Gan, Israel}
\email{aklerner@netvision.net.il}

 \address{Department of Mathematics, University of Alabama, Tuscaloosa, AL 35487-0350 }
\email{kmoen@as.ua.edu}



\thanks{The second author is partially supported by the NSF under grant 1201504}

\begin{abstract} We establish several mixed $A_p$-$A_\infty$ bounds for Calder\'on-Zygmund operators that only involve one supremum.  We address both cases when the $A_\infty$ part of the constant is measured using the exponential-logarithmic definition and using the Fujii-Wilson definition.  In particular, we answer a question of the first author and provide an answer, up to a logarithmic factor, of a conjecture of Hyt\"onen and Lacey.  Moreover, we give an example to show that our bounds with the logarithmic factors can be arbitrarily smaller than the previously known bounds (both one supremum and two suprema).

\end{abstract}

\maketitle

\section{Introduction}
 Hyt\"onen's \cite{Hy2011} recent solution of the $A_2$ conjecture, states that any Calder\'on-Zygmund operator satisfies the following bound on weighted Lebesgue spaces:
\begin{equation}\label{A2}\|T\|_{L^p(w)}\lesssim  [w]_{A_p}^{\max(1,\frac{1}{p-1})}.\end{equation}
Recently inequality \eqref{A2} has seen several improvements.  These come in the form of the so-called ``mixed estimates''.  The idea behind the mixed estimates is that one only needs the full strength of the $A_p$ constant for part of the estimates, while the other part only requires something weaker.   The smaller quantities come in the form of $A_r$ constants for large $r$ or $A_\infty$ constants.  Below we will attempt to describe these results.

First we require some terminology.  A weight will be a nonnegative locally integrable function.  Given a weight $w$, exponent $1< p<\infty$, and cube $Q$, define the precursor to the $A_p$ constant as
$$A_p(w,Q)=\Big(\,\dashint_Q w\,\Big)\Big(\,\dashint_Q w^{-\frac{1}{p-1}}\,\Big)^{p-1}=\frac{w(Q)\sigma(Q)^{p-1}}{|Q|^p}$$
where $\sigma=w^{-\frac{1}{p-1}}$.  When $p=1$ we define the limiting quantity as 
$$A_1(w,Q)=\Big(\,\dashint_Q w\,\Big)(\inf_Q w)^{-1}=\lim_{p\ra 1}A_p(w,Q).$$ 
For $p=\infty$ we will consider two constants.  The first constant is defined as a limit of the $A_p(w,Q)$ constants: 
$$A^{\exp}_\infty(w,Q)=\Big(\,\dashint_Q w\,\Big)\exp\Big(\,\dashint_Q \log w^{-1}\,\Big)=\lim_{p\ra \infty} A_p(w,Q).$$
For the second constant let 
$$A_\infty(w,Q)=\frac{1}{w(Q)}\int_Q M(w\chi_Q)\,$$
where $M$ is the Hardy-Littlewood maximal operator (see Section \ref{prelim}).  By Jensen's inequality we see that the quantities $A_p(w,Q)$ decrease as $p$ increases.
Define the following constants:
$$ [w]_{A_p}=\sup_Q A_p(w,Q),$$
$$[w]_{A^{\exp}_\infty}=\sup_Q A^{\exp}_\infty(w,Q),$$
and
$$[w]_{A_\infty}=\sup_{Q}A_\infty(w,Q).$$
We write $w\in A_p$ if $[w]_{A_p}<\infty$ and $w\in A_\infty$ if $[w]_{A_\infty^{\exp}}<\infty$ or $[w]_{A_\infty}<\infty$.  The constant $[w]_{A^{\exp}_\infty}$ was defined by Hru{\v{s}}{\v{c}}ev in \cite{MR727244}.  The constant $[w]_{A_\infty}$ was defined by Fujii \cite{MR0481968} and Wilson \cite{MR883661,MR2359017}, who also showed that both constants define the class $A_\infty$.  Hyt\"onen and P\'erez \cite{HP2011} proved the quantitative upper bound 
\begin{equation}\label{eqn:Ainftyrel}[w]_{A_\infty}\lesssim [w]_{A^{\exp}_\infty}\end{equation}
and provided examples to show that $[w]_{A^{\exp}_\infty}$ can be exponentially larger than $[w]_{A_\infty}$ (see also Beznosova and Reznikov \cite{BR2011}).  While inequality \eqref{eqn:Ainftyrel} holds, it is not clear what the relationship is between $A^{\exp}_\infty(w,Q)$ and $A_\infty(w,Q)$ for a fixed cube $Q$.  Hereafter we will refer to constants that contain a quantity depending on $A^{\exp}_\infty(w,Q)$ as exponential $A_\infty$ constants and constants depending on $A_\infty(w,Q)$ as simply $A_\infty$ constants. 
 
Let us now define the mixed type constants.  Given $1\leq p<\infty$ and real numbers $\al,\beta$ define the  mixed constants:
$$[w]_{(A_p)^\al (A_r)^\beta}=\sup_Q A_p(w,Q)^\al A_r(w,Q)^\beta, \quad 1\leq r\leq\infty,$$
and the exponential mixed constants:
$$[w]_{(A_p)^\al (A_\infty^{\exp})^\beta}=\sup_Q A_p(w,Q)^\al A_\infty^{\exp}(w,Q)^\beta.$$

Inequalities involving constants of the form
$$[w]_{(A_p)^\al(A_r)^\beta}, \ [w]_{(A_p)^\al(A_\infty)^\beta}, \  \text{or} \quad [w]_{(A_p)^\al(A_\infty^{\exp})^\beta}$$
will be said to be one supremum estimates.  Whereas estimates containing products of separate constants such as
$$[w]_{A_p}^\al [w]_{A_\infty}^\beta \quad \text{or} \qquad [w]_{A_p}^\al [w]_{A_\infty^{\exp}}^\beta$$
will be referred to as two suprema estimates.

 The pioneer work on mixed constants involving $A_\infty$ was done by Hyt\"onen and P\'erez \cite{HP2011}.  For the Hardy-Littlewood maximal operator they were able to prove two estimates.  The first, a one supremum estimate containing the exponential mixed constant.
 \begin{nonum}[\cite{HP2011}] If $1<p<\infty$ and $w\in A_p$ then
\begin{equation}\label{eqn:maxexp0} \|M\|_{L^p(w)}\lesssim [\sigma]_{(A_{p'})^{\frac{1}{p'}}(A^{\exp}_\infty)^{\frac{1}{p}}}.\end{equation}
\end{nonum}
Second, they prove the following two suprema estimate.
  \begin{nonum}[\cite{HP2011}] If $1<p<\infty$ and $w\in A_p$ then
\begin{equation}\label{eqn:maxW2}\|M\|_{L^p(w)}\lesssim [\sigma]_{A_{p'}}^{\frac{1}{p'}}[\sigma]_{A_\infty}^{\frac{1}{p}}=([w]_{A_p}[\sigma]_{A_\infty})^{\frac1p}.\end{equation}
\end{nonum}
Both of the inequalities \eqref{eqn:maxexp0} and \eqref{eqn:maxW2} improve Buckley's \cite{MR1124164} well known bound
$$\|M\|_{L^p(w)}\lesssim [\sigma]_{A_{p'}}=[w]_{A_p}^{\frac{1}{p-1}}.$$
A natural question is whether inequality \eqref{eqn:maxW2} can be replaced by a one supremum estimate.  In this vein we have our first result.  To state the following results we define the function
$$\Phi(t)=1+\log(t).$$
\begin{theorem} \label{thm:maxW} Suppose $1<p<\infty$ and $w\in A_p$, then
\begin{equation}\label{eqn:maxW} \|M\|_{L^p(w)}\lesssim \Phi([\sigma]_{A_{p'}})^{\frac1p}[\sigma]_{(A_{p'})^{\frac{1}{p'}}(A_\infty)^{\frac{1}{p}}}.\end{equation}
\end{theorem}
\noindent We do not know whether or not the logarithmic factor in \eqref{eqn:maxW} is necessary, that is, we do not know how to remove the logarithmic factor or find an example showing it is necessary.  

For Calder\'on-Zygmund operators much less is known about one supremum estimates.  Hyt\"onen and P\'erez proved a two suprema estimate when $p=2$ which was later extended to $1<p<\infty$, first for the Hilbert transform by Lacey \cite{L2011}, and then for general Calder\'on-Zygmund operators by Lacey and Hyt\"onen \cite{HL2011}. 
\begin{nonum}[\cite{HL2011,HP2011,L2011}] \label{thm:hlp} If $1<p<\infty$, $T$ is Calder\'on-Zygmund operator, and $w\in A_p$, then
\begin{equation}\label{eqn:LHPmixed}\|T\|_{L^p(w)}\lesssim [w]_{A_p}^{\frac1p}([w]_{A_\infty}^{\frac{1}{p'}}+[\sigma]_{A_\infty}^{\frac1p}).\end{equation}
\end{nonum}

Meanwhile, the first author examined weighted estimates with one supremum, where the smaller part was an $A_r$ constant from \cite{Lern2011}. 
\begin{nonum}[\cite{Lern2012,Lern2011}] \label{thmpr} If $1<p,r<\infty$, $T$ is a Calder\'on-Zygmund, and $w\in A_p$ then
\begin{equation}\label{eqn:mixedpr}\|T\|_{L^p(w)}\lesssim [w]_{(A_p)^{\frac{1}{p-1}}(A_r)^{1-\frac{1}{p-1}}}+[\sigma]_{(A_{p'})^{\frac{1}{p'-1}}(A_r)^{1-\frac{1}{p'-1}}}.\end{equation}
\end{nonum}
Notice that the right hand side of \eqref{A2} can be written as
$$[w]_{A_p}^{\max(1,\frac{1}{p-1})}\simeq [w]_{A_p}+[\sigma]_{A_{p'}}.$$
With this in mind it is easy to see that inequalities \eqref{eqn:LHPmixed} and \eqref{eqn:mixedpr} both improve \eqref{A2}, while explicit examples show that right hand sides of \eqref{eqn:LHPmixed} and \eqref{eqn:mixedpr} are incomparable.  We emphasize that the bounds \eqref{eqn:LHPmixed} and \eqref{eqn:mixedpr} differ twofold: the latter has one supremum constants and the smaller part of the mixed constant is an $A_r$ measurement for $1<r<\infty$.  Explicit examination of the proof of \eqref{eqn:mixedpr} in \cite{Lern2011} shows that one cannot take $r=\infty$ because of a factor of $2^r$ involved in the calculations.  The first author went on to ask if it was possible to take $r=\infty$ in inequality \eqref{eqn:mixedpr}.  Our first result for Calder\'on-Zygmund operators is a positive answer to this question.

\begin{theorem} \label{thm:exp1} If $1<p<\infty$ and $T$ is a Calder\'on-Zygmund operator then
\begin{equation}\label{eqn:exp1} \|T\|_{L^{p}(w)}\lesssim [w]_{(A_p)^{\frac{1}{p-1}}(A_\infty^{\exp})^{1-\frac{1}{p-1}}}.\end{equation}
\end{theorem}

The astute reader will notice that taking $r=\infty$ in inequality \eqref{eqn:mixedpr} should yield the sum of two constants, whereas inequality \eqref{eqn:exp1} only has one constant.  However they are equivalent since for these particular exponents,
$$[w]_{(A_p)^{\frac{1}{p-1}}(A^{\exp}_\infty)^{1-\frac{1}{p-1}}}=[\sigma]_{(A_{p'})^{\frac{1}{p'-1}}(A^{\exp}_\infty)^{1-\frac{1}{p'-1}}}.$$

Hyt\"onen and Lacey \cite{HL2011} went on to conjecture that one should be able to replace the right hand side of \eqref{eqn:LHPmixed} with an estimate containing one supremum constants.
\begin{conjecture}[\cite{HL2011}] \label{conj:ult1} If $1<p<\infty$, $T$ is a Calder\'on-Zygmund operator, and $w\in A_p$ then
\begin{equation}\label{eqn:W0} \|T\|_{L^p(w)}\lesssim [w]_{(A_p)^{\frac{1}{p}}(A_\infty)^{\frac{1}{p'}}}+[\sigma]_{(A_{p'})^{\frac{1}{p'}}(A_\infty)^{\frac{1}{p}}}.\end{equation}
\end{conjecture}
We are able to give a partial answer to Conjecture \ref{conj:ult1} and the corresponding version containing the exponential $A_\infty$.  However, our estimates contain an extra logarithmic factor.  In this vein, our first result is estimate containing the exponential mixed constants.  
\begin{theorem} \label{thm:exp0} Under the same hypothesis as Theorem \ref{thm:exp1} we have 
\begin{equation*} \label{eqn:weakexp0} \|T\|_{L^{p,\infty}(w)}\lesssim \Phi([w]_{A_p})^{\frac1p}[w]_{(A_p)^{\frac1p}(A^{\exp}_\infty)^{\frac{1}{p'}}}. \end{equation*}
and
\begin{equation*} \label{eqn:weakexp0} \|T\|_{L^{p}(w)}\lesssim \Phi([w]_{A_p})^{\frac1p}[w]_{(A_p)^{\frac1p}(A^{\exp}_\infty)^{\frac{1}{p'}}}+\Phi([\sigma]_{A_{p'}})^{\frac{1}{p'}}[\sigma]_{(A_p)^{\frac{1}{p'}}(A^{\exp}_\infty)^{\frac{1}{p}}}. \end{equation*}
\end{theorem}
 
 For the mixed $A_\infty$ constants we obtain a slightly worse power on the logarithmic factor.  
\begin{theorem} \label{thm:W0} Under the same hypothesis as Theorem \ref{thm:exp1} we have 
\begin{equation*} \|T\|_{L^{p,\infty}(w)}\lesssim \Phi([w]_{A_p})[w]_{(A_p)^{\frac1p}(A_\infty)^{\frac{1}{p'}}}.\end{equation*}
 and
 \begin{equation*} \|T\|_{L^{p}(w)}\lesssim \Phi([w]_{A_p})([w]_{(A_p)^{\frac1p}(A_\infty)^{\frac{1}{p'}}}+ [\sigma]_{(A_{p'})^{\frac{1}{p'}}(A_\infty)^{\frac{1}{p}}}).\end{equation*}
 \end{theorem}

\noindent Again we do not know if the logarithmic factors in Theorems \ref{thm:exp0}  or \ref{thm:W0} are necessary.

Finally we end with one last estimate that while having two suprema, is an improvement over several known results.  In \cite{LO2012JFA} the first author and Ombrosi conjecture that the following bound 
\begin{equation}\label{eqn:p<q}\|T\|_{L^p(w)}\lesssim [w]_{A_q}\end{equation}
should hold for $1<q<p<\infty$ and $w\in A_q \ (\subsetneq A_p)$.  Inequality \eqref{eqn:p<q} was proven by Duoandikoetxea in \cite{MR2754896} by means of extrapolation.  We make an observation that one may improve this bound by using the weak-type bound of Hyt\"onen and Lacey \cite{HL2011}
\begin{equation}\label{eqn:mixedweak2sup} \|T\|_{ L^{p,\infty}(w)}\lesssim [w]_{A_p}^{\frac1p}[w]_{A_\infty}^{\frac{1}{p'}}.\end{equation}
\begin{theorem} \label{thm:mixedp<q} If $1\leq q<p<\infty$, $T$ is a Calder\'on-Zygmund operator, and $w\in A_q$ then 
$$\|T\|_{L^p(w)}\lesssim [w]_{A_q}^{\frac1p} [w]_{A_\infty}^{\frac{1}{p'}}.$$
\end{theorem}
We also note that Theorem \ref{thm:mixedp<q} also improves the $A_1$ result from \cite{HP2011}
$$\|T\|_{L^p(w)}\lesssim [w]_{A_1}^{\frac1p} [w]_{A_\infty}^{\frac{1}{p'}}.$$
We believe that Theorem \ref{thm:mixedp<q} should hold for constants with one supremum, for example, by replacing $[w]_{A_q}^{\frac1p}[w]_{A_\infty}^{\frac{1}{p'}}$ with 
$$[w]_{(A_q)^{\frac1p}(A^{\exp}_\infty)^{\frac{1}{p'}}} \ \text{or}\ [w]_{(A_q)^{\frac1p}(A_\infty)^{\frac{1}{p'}}}.$$ 
Our methods do not yield this result.

The organization of the paper will be as follows.  In Section \ref{prelim} we will introduce the necessary material on Calder\'on-Zygmund operators, dyadic grids, sparse families of cubes, and testing conditions.  In Section \ref{maximal} we present the proof of Theorem \ref{thm:maxW}.  Section \ref{proofs} contains the proofs of our main results for Calder\'on-Zygmund operators, Theorems \ref{thm:exp1}, \ref{thm:exp0}, \ref{thm:W0}, and \ref{thm:mixedp<q}.  Finally we end the manuscript with some further examples, observations, and questions in Section \ref{Further}.

\section{Preliminaries}  \label{prelim}

Given a measurable set $E\subseteq \R^n$, $|E|$ will denote the Lebesgue measure of $E$.  We will simultaneously view weights as functions and measures, for example, $w(E)$ will denote the the weighted measure of $E$: $w(E)=\int_E w$.  The average of a function on a cube $Q$ will be denoted
$$\dashint_Q f=\frac{1}{|Q|}\int_Q f.$$  
Finally, we will use the notation $A\lesssim B$ to indicate that there is a constant $c$, independent of the important parameters, such that $A\leq cB$.  We will write $A\simeq B$ when $A\lesssim B$ and $B\lesssim A$.  All further notation will be standard or defined as needed.  
\subsection{The main operators} \label{operators} The Hardy-Littlewood maximal operator is given by
$$Mf(x)=\sup_{Q\ni x}\dashint_Q |f|.$$
We will also need the following variant, known as the geometric maximal operator,
$$M_0 f(x)=\sup_{Q\ni x}\exp\Big(\,\dashint_Q \log|f|\Big).$$
Geometric maximal operators have long been studied (see \cite{MR1628101} and the references therein).   For our purpose we will use the fact that
$$M_0:L^p(\R^n)\ra L^p(\R^n), \qquad 0<p<\infty$$
(for a proof see \cite{HP2011}).

A Calder\'on-Zygmund operator is an $L^2(\R^n)$ bounded operator associated to a kernel $K$ for functions with compact support by the equality
$$Tf(x)=\int_{\R^n}K(x,y)f(y)\,dy \qquad x\notin \text{supp} \ f,$$
where $K$ satisfies the standard size and smoothness estimates:
\begin{enumerate}[(a)]
\item $|K(x,y)|\lesssim |x-y|^{-n}$ for $x\not=y$,
\item $|K(x+h,y)-K(x,y)|+|K(x,y+h)-K(x,y)|\lesssim {|h|^\delta}{|x-y|^{-n-\delta}}$ for some $\delta\in (0,1]$ when $|x-y|\geq 2|h|$.
\end{enumerate}
One may also define the maximally truncated version:
$$T_\star f(x)=\sup_{\ep>0}\Big|\int_{|x-y|>\ep}K(x,y)f(y)\,dy\Big|.$$
We will use common notation for the operator norms:
$$\|S\|_{L^p(w)}=\sup_{\|f\|_{L^p(w)}=1}\|Sf\|_{L^p(w)}$$
and
$$\|S\|_{L^{p,\infty}(w)}=\sup_{\|f\|_{L^p(w)}=1}\|Sf\|_{L^{p,\infty}(w)}$$
where $S$ is a sub-linear operator.  Occasionally we will wish to explicitly state that the operator acts  between two different function spaces, in which case we will write
$$\|S\|_{X\ra Y}=\sup_{\|f\|_{X}=1}\|Sf\|_{Y}.$$
For example
$$\|S\|_{L^{p}(\sigma)\ra L^p(w)}=\sup_{\|f\|_{L^p(\sigma)}=1}\|Sf\|_{L^{p}(w)}.$$
\subsection{Dyadic grids and sparse families} \label{dyadiccubes}
A dyadic grid, usually denoted $\D$, is a collection of cubes in $\R^n$ with the following properties:
\begin{enumerate}[\indent (a)]
\item the side-length of each cube satisfies $\ell(Q)=2^k$ for some $k\in \Z$;
\item given $Q,P\in \D$, $Q\cap P\in \{P, Q, \varnothing\}$;
\item for a fixed $k\in \Z$ the set $\D_k=\{Q\in \D:\ell(Q)=2^k\}$ forms a partition of $\R^n$.
\end{enumerate}

Given a dyadic grid, $\D$, we define the dyadic maximal operator as
$$M^\D f(x) =\sup_{\substack{Q\in \D \\ Q\ni x}}\dashint_Q |f(y)|\,dy.$$
It is well known (see \cite{HP2011,Lern2012}) that the Hardy-Littlewood maximal operator is pointwise equivalent to the finite sum of dyadic maximal functions.  Specifically, there exists dyadic grids, $\D^1,\ldots, \D^N$ and a dimensional constant $c_n$ such that 
\begin{equation}\label{eqn:dyadmaxbd} Mf(x)\leq c_n \sum_{k=1}^N M^{\D^k}f(x),\end{equation}
(the opposite inequality is trivial).  Thus, when obtaining bounds for $M$ it suffices to work with $M^\D$ for general dyadic grid $\D$.  Moreover, it is often useful to change the measure from Lebsegue measure to a weighted measure.  Specifically, given a weight $w$ and a dyadic grid $\D$, define the dyadic maximal function with respect to $w$ by
$$M^\D_w f(x)=\sup_{\substack{Q\in \D \\ x\in Q}}\frac{1}{w(Q)}\int_Q |f|w.$$ 
 The maximal operator $M_w^\D$ satisfies the $L^p(w)$ bounds 
$$\|M_w^\D\|_{L^p(w)}\leq p'$$
(see \cite{Moen2012} for a proof).

Let $\D$ be a dyadic grid, $\Q\subset \D$, and for each $Q\in \D$ define
$$\Q(Q)=\{Q'\in \Q: Q'\subseteq Q\}, \ \text{and}\ \Q'(Q)=\{Q'\in \Q: Q'\subsetneq Q\}.$$
We say a collection $\Sp\subseteq \D$ is a sparse family or simply sparse if 
$$\Big| \bigcup_{Q'\in \Sp'(Q)} Q'\Big|\leq \frac12|Q|, \qquad Q\in \Sp.$$
For each $Q\in \Sp$ define
$$E_Q=Q\backslash \Big( \bigcup_{Q'\in \Sp'(Q)} Q'\Big).$$
Then $\{E_Q\}_{Q\in \Sp}$ is a pairwise disjoint family that satisfies: 
$$\frac12|Q|\leq |E_Q|\leq |Q|.$$
Given a sparse family $\Sp$, we define the sparse operators
$$M^\Sp f=\sum_{Q\in \Sp}\, \Big(\,\dashint_Q f\,\Big) \cdot \chi_{E_Q}$$
and
$$T^\Sp f=\sum_{Q\in \Sp}\, \Big(\,\dashint_Q f\,\Big) \cdot \chi_Q.$$
The difference between the operators $M^\Sp$ and $T^\Sp$ is that the characteristic functions in the definition of $M^\Sp$ are over the pairwise disjoint family $\{E_Q\}_{Q\in \Sp}$.  Given a function $f$ that is bounded with compact support, by analyzing the level sets of $M^\D f$ one can prove that there exists a sparse family $\Sp=\Sp(f)$ such that
\begin{equation}\label{eqn:sparsemax} M^\D f\simeq M^\Sp f\end{equation}
where the implicit constants depend only on the dimension, not $\Sp$ or $f$.  The equivalence \eqref{eqn:sparsemax} can be traced back to Sawyer's characterization of two weight inequalities for the Hardy-Littlewood maximal operator \cite{MR676801}.  

It turns out that sparse operators dominate Calder\'on-Zygmund operators as well.  The following Theorem was proven by the first author.
\begin{nonum}[\cite{Lern2012},\cite{1Lern2012}] \label{thm:lernsparse} Suppose $T$ is a Calder\'on-Zygmund operator, $T_\star$ is the maximally truncated version and $X$ is any Banach function space (see \cite[Chapter 1]{MR928802}) then
$$\|T \|_X\leq c_T \sup_{\Sp}\|T^\Sp \|_{X}\quad
\text{and} \quad \|T_\star \|_X\leq c_T \sup_{\Sp}\|T^\Sp \|_{X}$$
where the suprema are over all sparse families of dyadic cubes.
\end{nonum}

By Theorem \ref{thm:lernsparse} we see that it suffices to work with a general dyadic grid $\D$ and sparse operator $T^\Sp$ to prove bounds for a Calder\'on-Zygmund operators.  Moreover, any bound that holds for sparse operators also holds for maximally truncated Calder\'on-Zygmund operators, thus, all of our results are valid for $T_\star$ as well.  To prove bounds for $T^\Sp$ we will use two weight testing conditions.  Given a sparse family $\Sp\subseteq \D$ and dyadic cube $R\in \D$, recall that $$\Sp(R)=\{Q\in \Sp: Q\subseteq R\}$$ and
$$T^{\Sp(R)} f=\sum_{\substack{Q\in \Sp(R)}}\Big(\,\dashint_Q f\,\Big)\cdot\chi_Q.$$
For $1<p<\infty$ and a pair of weights $(u,\sigma)$ define the testing constant
$$[w,\sigma]_{T^\Sp_{p}}=\sup_{R\in \D} \sigma(R)^{-1/p}\Big(\,\int_R (T^{\Sp(R)}\sigma)^pw\,dx\Big)^{1/p}.$$
We will also need the dual testing constant, $[\sigma,w]_{T^\Sp_{p'}}$, formed by interchanging the roles of $w$ with $\sigma$ and $p$ with $p'$.  We have the following Theorem of Lacey, Sawyer, and Uriarte-Tuero.
\begin{nonum}[\cite{LacSawUT2010}] \label{testing} Suppose $1<p<\infty$, $\D$ is a dyadic grid, $\Sp$ is a sparse subset of $\D$, and $(w,\sigma)$ is a pair of weights then the following equivalences hold
\begin{equation*} \|T^\Sp(\,\cdot\,\sigma)\|_{L^p(\sigma)\ra L^{p,\infty}(w)}\simeq [\sigma,w]_{T^\Sp_{p'}}\end{equation*}
\begin{equation*}\|T^\Sp(\,\cdot\,\sigma)\|_{L^p(\sigma)\ra L^{p
}(w)}\simeq[w,\sigma]_{T^\Sp_{p}}+ [\sigma,w]_{T^\Sp_{p'}}.\end{equation*}
\end{nonum}
Thus to prove Theorems \ref{thm:exp1}, \ref{thm:exp0}, and \ref{thm:W0} we simply estimate the constant $[\sigma,w]_{T^\Sp_{p'}}$.  The following theorems for sparse operators imply Theorems \ref{thm:exp1}, \ref{thm:exp0}, and \ref{thm:W0}.
\begin{theorem}\label{thm:testexp1} If $1<p<\infty$, $\Sp\subseteq  \D$ is a sparse family of dyadic cubes, $w\in A_p$ with $\sigma=w^{1-p'}$, then
$$[\sigma,w]_{T^\Sp_{p'}}\lesssim [w]_{(A_p)^{\frac{1}{p-1}}(A^{\exp}_\infty)^{1-\frac{1}{p-1}}}.$$
\end{theorem}
 \begin{theorem}\label{thm:testexp0} Under the same hypothesis as Theorem \ref{thm:testexp1} we have
$$[\sigma,w]_{T^\Sp_{p'}}\lesssim \Phi([w]_{A_p})^{\frac1p}[w]_{(A_p)^{\frac{1}{p}}(A^{\exp}_\infty)^{\frac{1}{p'}}}.$$
\end{theorem}
 \begin{theorem}\label{thm:testW0} Under the same hypothesis as Theorem \ref{thm:testexp1} we have
$$[\sigma,w]_{T^\Sp_{p'}}\lesssim \Phi([w]_{A_p})[w]_{(A_p)^{\frac{1}{p}}(A_\infty)^{\frac{1}{p'}}}.$$
\end{theorem}

We note that Theorems \ref{thm:testexp1}, \ref{thm:testexp0}, and \ref{thm:testW0} easily imply the corresponding weak type bound in \ref{thm:exp1}, \ref{thm:exp0}, and \ref{thm:W0} respectively, since
$$\|T\|_{L^{p,\infty}(w)}=\|T(\,\cdot\,\sigma)\|_{L^p(\sigma)\ra L^{p,\infty}(w)}\lesssim \sup_{\Sp\subseteq \D} \|T^\Sp(\,\cdot\,\sigma)\|_{L^p(\sigma)\ra L^p(w)}.$$
They also imply the strong type bounds by symmetry.  Indeed, by interchanging the roles of $w$ with $\sigma$ and $p$ with $p'$ we see that
$$[w,\sigma]_{T^\Sp_{p}}\lesssim [\sigma]_{(A_{p'})^{\frac{1}{p'-1}}(A^{\exp}_\infty)^{1-\frac{1}{p'-1}}},$$
$$[w,\sigma]_{T^\Sp_{p}}\lesssim \Phi([\sigma]_{A_{p'}})^{\frac{1}{p'}}[\sigma]_{(A_{p'})^{\frac{1}{p'}}(A^{\exp}_\infty)^{\frac{1}{p}}},$$
and
$$[w,\sigma]_{T^\Sp_{p}}\lesssim \Phi([\sigma]_{A_{p'}})[\sigma]_{(A_{p'})^{\frac{1}{p'}}(A_\infty)^{\frac{1}{p}}} \simeq\Phi([w]_{A_{p}})[\sigma]_{(A_{p'})^{\frac{1}{p'}}(A_\infty)^{\frac{1}{p}}}.$$

\section{Mixed estimates for the Hardy-Littlewood maximal operator} \label{maximal}
In this section we prove Theorem \ref{thm:maxW}.  Our techniques will strongly parallel those for Calder\'on-Zygmund operators.  Define the testing constant
$$[u,\sigma]_{S_p}^p=\sup_R\frac{\int_R M(\chi_R\sigma)^pw}{\sigma(R)}.$$
It was shown by the second author in \cite{MR2534183} that
\begin{equation}\label{eqn:testmax} \|M\|_{L^p(w)}\simeq [u,\sigma]_{S_p}.\end{equation}
\begin{proof}[Proof of Theorem \ref{thm:maxW}] 
By combining \eqref{eqn:dyadmaxbd}, \eqref{eqn:sparsemax}, and \eqref{eqn:testmax} we see that it suffices to estimate the quantity
\begin{equation}\label{eqn:maxbilinear} \int_RM^\Sp(\chi_R\sigma)^pw\,dx=\sum_{Q\in\Sp(R)}\Big(\frac{\sigma(Q)}{|Q|}\Big)^p w\big(E_Q\big)\end{equation}
for a sparse family $\Sp\subseteq \D$.  Note that we have used the disjointness of the family $\{E_Q\}_{Q\in \Sp}$ in the equality \eqref{eqn:maxbilinear}.  For $a\in \Z$ define
$$\Q^a=\{ Q\in \Sp: 2^{a-1}<A_p(w,Q)\leq 2^a\}.$$
Then $\Q^a$ is empty if $a> \log_2[w]_{A_p}$ or $a< -1$.  Set $K=\lfloor \log_2[w]_{A_p}\rfloor$.
Then,
$$\Sp(R)=\bigcup_{a=-1}^K \Q^a,$$
and the sum in \eqref{eqn:maxbilinear} is bounded by
\begin{align}\sum_{Q\in\Sp(R)}\Big(\frac{\sigma(Q)}{|Q|}\Big)^p w\big(E_Q\big)&\leq \sum_{a=-1}^K\sum_{Q\in \Q^a}\frac{w(Q)}{|Q|}\Big(\frac{\sigma(Q)}{|Q|}\Big)^p|Q|.\nonumber\\
&\leq \sum_{a=-1}^K2^a\sum_{Q\in \Q^a}\sigma(Q).\label{Qasum}
\end{align}
Let $\Q^a_{\max}$ be the collection of maximal cubes in $\Q^a$, then
\begin{align*}
\sum_{Q\in \Q^a}\sigma(Q)&=\sum_{Q\in \Q^a_{\max}}\sum_{\substack{P\in \Q^a\\ P\subseteq Q}}\sigma(P)\\
&\simeq\sum_{Q\in \Q^a_{\max}}\sum_{\substack{P\in \Q^a\\ P\subseteq Q}}\frac{\sigma(P)}{|P|}|E_P|\\
&\leq \sum_{Q\in \Q^a_{\max}}\int_{Q}M(\chi_Q\sigma).\,
\end{align*}
Substituting this back into \eqref{Qasum} we arrive at
\begin{align*}\int_RM^\Sp(\chi_R\sigma)^pw\,dx &\lesssim \sum_{a=-1}^K\sum_{Q\in \Q^a}\frac{w(Q)}{|Q|}\Big(\frac{\sigma(Q)}{|Q|}\Big)^p|Q|.\\
&\leq \sum_{a=-1}^K2^a\sum_{Q\in \Q^a}\sigma(Q).\\
&\leq \sum_{a=-1}^K2^a\sum_{Q\in \Q_{\max}^a}\int_{Q}M(\chi_Q\sigma)\,dx.\\
&\leq \sum_{a=-1}^K\sum_{Q\in \Q_{\max}^a}A_{p'}(\sigma,Q)^{\frac{p}{p'}}A_\infty(\sigma,Q) \sigma(Q)\\
&\leq [\sigma]_{(A_{p'})^{\frac{1}{p'}}(A_\infty)^{\frac1p}}^p\sum_{a=-1}^K\sum_{Q\in \Q_{\max}^a}\sigma(Q)\\
&\lesssim[\sigma]_{(A_{p'})^{\frac{1}{p'}}(A_\infty)^{\frac1p}}^p(1+K)\sigma(R)\\
&\simeq[\sigma]_{(A_{p'})^{\frac{1}{p'}}(A_\infty)^{\frac1p}}^p\Phi([\sigma]_{A_{p'}})\sigma(R).
\end{align*}

\end{proof}

\section{Mixed bounds for sparse operators} \label{proofs}
	We now prove our main results for sparse operators, Theorems \ref{thm:testexp1}, \ref{thm:testexp0}, and \ref{thm:testW0}, which, as mentioned above, imply Theorems \ref{thm:exp1}, \ref{thm:exp0}, and \ref{thm:W0} respectively.  In all of the estimates we aim to bound the testing constant
$$[\sigma,w]_{T^\Sp_{p'}}^{p'}=\sup_R\frac{\int_R (T^{\Sp(R)}w)^{p'}\sigma\,dx}{w(R)}.$$	
We begin with the proof of Theorem \ref{thm:testexp1} because it requires different machinery than Theorems \ref{thm:testexp0} and \ref{thm:testW0}.  Specifically, we are able to prove Theorem \ref{thm:testexp1} without a corona decomposition, while our proofs of the other results require this tool.
	
\begin{proof}[Proof of Theorem \ref{thm:testexp1}] It suffices to show that for a dyadic grid $\D$, sparse subset $\Sp$, and fixed cube $R$ that
$$\int_R\Big( \sum_{Q\in \Sp(R)} \Big(\,\dashint_Q w\,\Big)\chi_Q\Big)^{p'}\sigma\,dx\lesssim [w]_{(A_p)^{\frac{1}{p-1}}(A_\infty^{\exp})^{1-\frac{1}{p-1}}}^{p'}w(R).$$
Multiplying and dividing by the expression defining $[w]_{(A_p)^{\frac{1}{p-1}}(A_\infty^{\exp})^{1-\frac{1}{p-1}}}$, we have that the problem reduces to showing
$$\int_R\Big( \sum_{Q\in \Sp(R)}\frac{|Q|}{\sigma(Q)} \Big(\exp\big(\,\dashint_Q \log w\,\big)\Big)^{1-\frac{1}{p-1}}\chi_Q\Big)^{p'}\sigma\,\lesssim w(R)$$
with the implicit constant independent of $w$.  By duality this is equivalent to showing for $\|g\|_{L^p(\sigma)}=1$ that
\begin{equation}\label{eqn:reducesumbd} \sum_{Q\in \Sp(R)}\frac{|Q|}{\sigma(Q)} \Big(\exp\big(\,\dashint_Q \log w\,\big)\Big)^{1-\frac{1}{p-1}}\int_Qg\sigma\,\lesssim w(R)^{1/p'}.\end{equation}
We work with the sum in \eqref{eqn:reducesumbd}, by H\"older's inequality we have 
\begin{align*}
\lefteqn{\sum_{Q\in \Sp(R)}\frac{|Q|}{\sigma(Q)} \Big(\exp\big(\,\dashint_Q \log w\,\big)\Big)^{1-\frac{1}{p-1}}\int_Qg\sigma\,}\\
&\leq \Big(\sum_{Q\in \Sp(R)}\Big(\frac{1}{\sigma(Q)}\int_Qg\sigma\,\Big)^p \Big(\exp\big(\,\dashint_Q \log w\,\big)\Big)^{1-p'}|Q|\Big)^{\frac1p}\\
&\qquad \times \Big(\sum_{Q\in \Sp(R)}\exp\big(\,\dashint_Q \log w\,\big)|Q|\Big)^{\frac{1}{p'}}.
\end{align*}
The second factor satisfies
\begin{align*}{\Big(\sum_{Q\in \Sp(R)}\exp\big(\,\dashint_Q \log w\,\big)|Q|\Big)^{\frac{1}{p'}}}\lesssim\lefteqn{\Big(\sum_{Q\in \Sp(R)}\exp\big(\,\dashint_Q \log w\,\big)|E_Q|\Big)^{\frac{1}{p'}}} \\
&\leq \Big(\int_R M_0(w\chi_R)\,\Big)^{\frac{1}{p'}}\lesssim w(R)^{1/p'}.
\end{align*}
Meanwhile, observing that
$$\Big(\exp\big(\,\dashint_Q \log w\,\big)\Big)^{1-p'}=\exp\big(\,\dashint_Q \log \sigma\,\big)$$
we see that the first multiple satisfies
\begin{align*}
\lefteqn{\Big(\sum_{Q\in \Sp(R)}\Big(\frac{1}{\sigma(Q)}\int_Qg\sigma\,\Big)^p \Big(\exp\big(\,\dashint_Q \log w\,\big)\Big)^{1-p'}|Q|\Big)^{\frac1p}}\\
&\qquad\lesssim \Big(\sum_{Q\in \Sp(R)} \exp\big(\,\dashint_Q \log\big( (M^\D_\sigma g)^p\sigma\big)\,\big)|E_Q|\Big)^{\frac1p}\\
&\qquad\lesssim \Big(\int_R M_0\big((M_\sigma^\D g)^p\sigma\big)\,\Big)^{1/p}.\\
\end{align*}
Since the $M_0$ is bounded on $L^1(\R^n)$ and $M_\sigma^\D$ is bounded on $L^p(\sigma)$ we see that this concludes the proof of our theorem.
\end{proof}

To prove Theorems \ref{thm:testexp0} and \ref{thm:testW0} we will use what by now has become a standard technique, a decomposition of dyadic operators known as a corona decomposition.    Similar decompositions can be found in the works \cite{CM2012,Hy2011,HP2011, HPTV2010,L2011,MR2657437}.  We provide a detailed proof of our corona decomposition noting that sparse families of cubes simplify some of the calculations.  In order to state it we need to define a maximal function.  Suppose $\Q$ is a family of dyadic cubes and $a=\{a_Q\}_{Q\in \Q}$ is a sequence indexed by the members of $\Q$, define the maximal function $M^\Q a=\sup_{Q\in \Q} |a_Q|\chi_Q.$   We have the following lemma.  
\begin{lemma}[Corona decomposition] \label{lem:corona} Suppose $R$ is a cube, $\Q$ is a collection of sparse cubes contained in $R$, $1< p<\infty$, $\nu$ is a Borel measure on $R$, and $a=\{a_Q\}_{Q\in \Q}$ is a sequence of positive constants that satisfy the following: 
\begin{enumerate}[(a)] \item $M^\Q a$ is finite almost everywhere on $R$,
\item there exist constants $c,C,r>0$ such that
\begin{equation}\label{eqn:balance} c\leq (a_Q )^r\frac{\nu(Q)}{|Q|}\leq C \qquad  \qquad Q\in \Q.\end{equation}
\end{enumerate}
Then there exists a sub-collection of cubes, $\Ca\subseteq \Q$, called the corona decomposition of $\Q$, such that the following inequality holds:
$$\left(\int_R\Big(\,\sum_{Q\in \Q}a_Q\cdot\chi_Q\Big)^p\,d\nu\right)^{\frac1p}\lesssim \frac{C}{c} \left(\sum_{Q\in \Ca}(a_Q)^p\nu(Q)\right)^{\frac1p}.$$
\end{lemma}
\begin{proof} Let $C^0$ denote the collection of all maximal cubes in $\Q$ and define $C^k$ for $k>1$ inductively as follows: $Q\in C^k$ if and only if the following three criteria are satisfied
\begin{enumerate}
\item there exists $P\in C^{k-1}$ containing $Q$,
\item the inequality
\begin{equation}\label{stopineq}a_Q>2\cdot a_P\end{equation}
holds, 
\item and $Q$ is maximal with respect to inclusion in $\Q$.  
\end{enumerate}
We note that when $k>1$, $C^k$ could possibly be empty.  Set $\Ca=\bigcup_{k}C^k$.  By the maximality of the stopping cubes, given any $Q\in \Q$ there exists a smallest $P\in \Ca$ such that $P\supseteq Q$; we denote such $P$ by $\Pi(Q)$.  Notice that the opposite inequality to \eqref{stopineq} must hold for $Q$ and $\Pi(Q)$, that is,
$$a_Q\leq 2\cdot a_{\Pi(Q)}.$$
For $P\in \Ca$ let
$$\Q(P)=\{Q\in \Q: \Pi(Q)=P\}.$$
Given $Q\in \Q(P)$ we now fix the ratio between $a_Q$ and $a_P$: for $b=0,1,2,\ldots$ and $P\in \Ca$ let $\Q^b(P)$ be all $Q\in \Q(P)$ such that
\begin{equation}\label{freezeavg}2^{-b}a_P < a_Q\leq 2^{-b+1}a_P.\end{equation}
Then, 
\begin{align*}\sum_{Q\in \Q}a_Q\cdot\chi_Q&=\sum_{P\in \Ca}\sum_{b=0}^\infty \sum_{Q\in \Q^b(P)}a_Q\cdot\chi_Q\\
&\leq\sum_{b=0}^\infty2^{-b}\sum_{P\in \Ca}a_P \sum_{Q\in \Q^b(P)}\chi_Q.
\end{align*}
For $k\geq 0$ define the sets 
$$E_k(P)=\Big\{x\in R:k<\sum_{Q\in \Q^b(P)}\chi_Q\leq k+1\Big\}$$
and
$$\Omega_k(P)=\Big\{x\in R:\sum_{Q\in \Q^b(P)}\chi_Q> k\Big\}.$$
We may further decompose the sum:
\begin{align*}\lefteqn{\sum_{b=0}^\infty2^{-b}\sum_{P\in \Ca}a_P \sum_{Q\in \Q^b(P)}\chi_Q}\\
&\leq \sum_{b=0}^\infty2^{-b}\sum_{k=0}^\infty (k+1)\sum_{P\in \Ca}a_P\cdot \chi_{E_k(P)}\\
&\leq \sum_{b=0}^\infty2^{-b}\sum_{k=0}^\infty (k+1)\sum_{P\in \Ca}a_P \cdot\chi_{\Omega_k(P)}.
\end{align*}
By Minkowski's inequality we have
\begin{align}\lefteqn{\Big(\int_R\Big(\sum_{Q\in \Q}a_Q\cdot\chi_Q\Big)^p\,d\nu\Big)^{\frac1p}}\nonumber \\ &\qquad \leq\sum_{b=0}^\infty2^{-b}\sum_{k=0}^\infty(k+1)\Big(\int_R\Big(\sum_{P\in \Ca}a_P\cdot \chi_{\Omega_k(P)}\Big)^p\,d\nu\Big)^{\frac1p}\label{sumbound}.\end{align}
Fix $x\in \{M^\Q a<\infty\}$.  Since $M^\Q a$ is finite a.e., there are at most finitely many stopping cubes that contain $x$.  Let $P_0\subseteq \cdots \subset P_m$ be the stopping cubes such that $x\in \Omega_k(P_i)\subseteq P_i$.  By construction, we have 
$$a_{P_i}< 2^{-i}a_{P_0}, \qquad i=1,\ldots,m.$$
For such $x$ we have
$$\Big(\sum_{P\in \Ca}a_P\cdot \chi_{\Omega_k(P)}(x)\Big)^p=\Big(\sum_{i=0}^ma_{P_i}\Big)^p<2^p (a_{P_0})^p\leq 2^p\sum_{P\in \Ca} (a_P)^p\chi_{\Omega_k(P)}(x).$$
Thus we may move the power $p> 1$ inside the innermost sum in \eqref{sumbound} to arrive at the bound
\begin{equation}\label{eqn:sumbd} \sum_{b=0}^\infty2^{-b}\sum_{k=0}^\infty(k+1)\Big(\sum_{P\in \Ca}(a_P)^p \,\nu\big(\Omega_k(P)\big)\Big)^{\frac1p}.\end{equation}
Finally notice that for each $k$ the set 
$$\Omega_k(P)=\Big\{x\in R:\sum_{Q\in \Q^b(P)}\chi_Q> k\Big\}=\bigcup_{j} Q^k_j$$
where for each $\{Q^k_j\}_j$ is a family of disjoint dyadic cubes in $\Q^b(P)$.  The cubes $\{Q_j^k\}$ are defined as follows.  Let $\{Q^0_j\}_j$ be the collection of all maximal cubes in $\Q^b(P)$ and define $\{Q^{k+1}_j\}_j$ inductively as those cubes that are maximal with respect to inclusion in $\Q^b(P)$ and contained in some $Q^k_{j}$.  By the sparsity condition we have that  
$$|\Omega_k(P)|\leq 2^{-k}|\Omega_1(P)|.$$  
For each $Q^k_j$, by combining \eqref{eqn:balance} and \eqref{freezeavg}, we have
$$c|Q^k_j|\big(2^{1-b} a_P\big)^{-r}\leq \nu(Q^k_j)\leq C\big(2^{-b}a_P\big)^{-r}|Q^k_j|$$
which implies
\begin{align*}\nu(\Omega_{k}(P))&=\sum_j \nu(Q^k_j)\leq C\big(2^{-b}a_P\big)^{-r}\sum_j|Q^k_j|\\
&= C\big(2^{-b}a_P\big)^{-r}|\Omega_k(P)|\leq 2^{-k}C\Big(2^{-b}\frac{|P|}{\mu(P)}\Big)^r|\Omega_1(P)|\\
&= 2^{-k}C\Big(2^{-b}a_P\Big)^{-r}\sum_j|Q^1_j|\lesssim 2^{-k}\frac{C}{c}\sum_j\nu(Q^1_j)\\&\leq 2^{-k}\frac{C}{c}\nu(P).
\end{align*}
Substituting this inequality into the sum \eqref{eqn:sumbd} we are able to sum in $k$ to arrive at the desired bound.
  \end{proof}

\begin{proof}[Proof of Theorem \ref{thm:testexp0}]

For the proof fix a cube $R$ and recall that $\Sp(R)=\{Q\in \Sp:Q\subseteq R\}$.  We aim to show that
\begin{equation*}\Big(\int_R\Big( \sum_{Q\in \Sp(R)} \Big(\,\dashint_Q w\Big)\chi_Q\Big)^{p'}\sigma\Big)^{\frac{1}{p'}}\lesssim  \Phi([w]_{A_p})^{\frac{1}{p}}[w]_{(A_p)^{\frac{1}{p}}(A^{\exp}_\infty)^{\frac{1}{p'}}}w(R)^{\frac{1}{p'}}.\end{equation*}

We first freeze $A_{p}$ constant.  Given $a\in \Z$ define
\begin{align*}\Q^a&:=\Big\{Q\in \Sp(R): 2^{a}<\Big(\,\dashint_Q w\,dx\Big)\Big(\,\dashint_Q \sigma\,dx\Big)^{p-1}\leq 2^{a+1}\Big\}\end{align*}
i.e., $\Q^a$ is all cubes in $\Sp(R)$ with $A_{p}(w,Q)\simeq 2^a$.  Notice that $\Q^a$ is empty if $a>\log_2[w]_{A_{p}}$ or $a<-1$.   As in the proof of Theorem \ref{thm:maxW} set $K=\lfloor \log_2[w]_{A_p}\rfloor$.  Then we have
$$\sum_{Q\in \Sp(R)}\Big(\,\dashint_Q w\Big)\chi_Q=\sum_{a=-1}^K \sum_{Q\in \Q^a} \Big(\,\dashint_Q w\Big)\chi_Q.$$
We now use Lemma \ref{lem:corona} to perform a corona decompositions of the sets $\Q^a$ with respect to the measure $\sigma$ and sequence $a_Q=\dashint_Q w$, $Q\in \Q^a$.  We have that there exists subset collections $\Ca^a$ of $\Q^a$ such that
\begin{align*}
\lefteqn{\Big(\int_R\Big( \sum_{Q\in \Sp(R)} \Big(\,\dashint_Q w\Big)\chi_Q\Big)^{p'}\sigma\Big)^{\frac{1}{p'}}\lesssim \sum_{a=-1}^K\Big(\sum_{Q\in \Ca^a}\Big(\,\dashint_Q w\,\Big)^{p'}\sigma(Q)\Big)^{\frac{1}{p'}}}\\
&\qquad\qquad\qquad\qquad\leq (K+1)^{\frac1p}\Big(\sum_{a=-1}^K\sum_{Q\in \Ca^a}\Big(\,\dashint_Q w\,\Big)^{p'}\sigma(Q)\Big)^{\frac{1}{p'}}\\
&\qquad \qquad\qquad\qquad\lesssim \Phi([w]_{A_p})^{\frac1p}\Big(\sum_{Q\in \Sp(R)}\Big(\,\dashint_Q w\,\Big)^{p'}\sigma(Q)\Big)^{\frac{1}{p'}}
\end{align*}
where in the second inequality we have used H\"older's inequality.   We now easily estimate the remaining sum

\begin{align*}\sum_{Q\in \Sp(R)}\Big(\,\dashint_Q w\,\Big)^{p'}\sigma(Q) &\leq [w]_{(A_p)^{\frac1p}(A_\infty^{\exp})^{\frac{1}{p'}}}^{p'}\sum_{Q\in \Sp(R)}\exp\big(\,\dashint_Q\log w\big)|Q|\\
&\lesssim [w]_{(A_p)^{\frac1p}(A^{\exp}_\infty)^{\frac{1}{p'}}}^{p'}\int_RM_0(w\chi_R)\\
&\lesssim[w]_{(A_p)^{\frac1p}(A^{\exp}_\infty)^{\frac{1}{p'}}}^{p'}w(R).
\end{align*}

\end{proof}
We now give a brief sketch of the proof of Theorem \ref{thm:testW0}.

\begin{proof}[Proof of Theorem \ref{thm:testW0}] With the same notation and reasoning as in the proof of Theorem \ref{thm:testexp0} we arrive at the estimate
\begin{align*}
\Big(\int_R\Big( \sum_{Q\in \Sp(R)} \Big(\,\dashint_Q w\Big)\chi_Q\Big)^{p'}\sigma\Big)^{\frac{1}{p'}}&\lesssim \sum_{a=-1}^K\Big(\sum_{Q\in \Ca^a}\Big(\,\dashint_Q w\,\Big)^{p'}\sigma(Q)\Big)^{\frac{1}{p'}}.\\
\end{align*}
Let $\Ca^a_{\max}$ be the collection of maximal cubes in $\Ca^a$, then
\begin{align*}
\sum_{a=-1}^K\Big(\sum_{Q\in \Ca^a}\Big(\,\dashint_Q w\,\Big)^{p'}\sigma(Q)\Big)^{\frac{1}{p'}}&\leq\sum_{a=-1}^K\Big(2^{a\frac{p'}{p}}\sum_{Q\in \Ca^a}\frac{w(Q)}{|Q|}|Q|\Big)^{\frac{1}{p'}} \\
&\leq\sum_{a=-1}^K\Big(2^{a\frac{p'}{p}}\sum_{Q\in \Ca_{\max}^a}\sum_{\substack{P\in \Ca^a\\ P\subseteq Q}}\frac{w(P)}{|P|}|P|\Big)^{\frac{1}{p'}} \\
&\leq\sum_{a=-1}^K\Big(2^{a\frac{p'}{p}}\sum_{Q\in \Ca_{\max}^a}\int_QM(\chi_Qw)\Big)^{\frac{1}{p'}} \\
&\lesssim [w]_{(A_p)^{\frac1p}(A_\infty)^{\frac{1}{p'}}}\sum_{a=-1}^K\Big(\sum_{Q\in \Ca_{\max}^a}w(Q)\Big)^{\frac{1}{p'}} \\
&\lesssim \Phi([w]_{A_p})[w]_{(A_p)^{\frac1p}(A_\infty)^{\frac{1}{p'}}}w(R)^{\frac{1}{p'}}.
\end{align*}
\end{proof}

Finally we end with the short observation that proves Theorem \ref{thm:mixedp<q}.
\begin{proof}[Proof of Theorem \ref{thm:mixedp<q}] Suppose $1<q<p<\infty$ and let $\ep=p-q$.  By inequality \eqref{eqn:mixedweak2sup} we have
$$\|T\|_{L^{p+\ep,\infty}(w)}\lesssim [w]_{A_{p+\ep}}^{\frac{1}{p+\ep}}[w]_{A_\infty}^{\frac{1}{(p+\ep)'}}\leq[w]_{A_{q}}^{\frac{1}{p+\ep}}[w]_{A_\infty}^{\frac{1}{(p+\ep)'}} $$
and
$$\|T\|_{ L^{p-\ep,\infty}(w)}\lesssim [w]_{A_{p-\ep}}^{\frac{1}{p-\ep}}[w]_{A_\infty}^{\frac{1}{(p-\ep)'}}=[w]_{A_{q}}^{\frac{1}{p-\ep}}\|w\|_{A_\infty}^{\frac{1}{(p-\ep)'}}.$$
By the Marcinkiewicz interpolation theorem we have
\begin{align*}\|T\|_{L^p(w)}&\lesssim \Big(\frac{2}{\ep}\Big)^{1/p}\|T\|^\theta_{ L^{p-\ep,\infty}(w)}\|T\|^{1-\theta}_{L^{p+\ep,\infty}(w)}
\end{align*}
where \begin{equation}\label{eqn:interp}\frac1p=\frac{\theta}{p-\ep}+\frac{1-\theta}{p+\ep}.\end{equation}  Using the weak bounds for $p-\ep$ and $p+\ep$ we have
\begin{align*}\|T\|_{L^p(w)\ra L^p(w)}&\lesssim (p-q)^{-1/p}[w]_{A_{q}}^{\frac{\theta}{p+\ep}}[w]_{A_\infty}^{\frac{\theta}{(p+\ep)'}}[w]_{A_{q}}^{\frac{1-\theta}{p-\ep}}[w]_{A_\infty}^{\frac{1-\theta}{(p-\ep)'}}.
\end{align*}
By the relationship of $\theta$ the powers on $[w]_{A_q}$ are
$$\frac{\theta}{p-\ep}+\frac{1-\theta}{p+\ep}=\frac{1}{p}$$
and the powers on $[w]_{A_\infty}$ are
$$\frac{\theta}{(p-\ep)'}+\frac{1-\theta}{(p+\ep)'}=1-\frac{1}{p}=\frac{1}{p'}.$$
\end{proof}

\section{Further questions and examples}\label{Further}
In this Section we observe some facts about the mixed constants and note that our bounds in Theorem \ref{thm:exp0} can be significantly smaller than both of the bounds in Theorems \ref{thm:hlp} and \ref{thm:exp1}.  First we make some observations about the behavior of the one supremum constants.  If $\al>0$ the class of weights satisfying
$$[w]_{(A_p)^\al(A^{\exp}_\infty)^\beta}<\infty,\ \ \text{or}\ \ [w]_{(A_p)^\al(A_\infty)^\beta}<\infty,$$
is simply $A_p$, since
$$\max([w]_{A_p}^\al,[w]_{A_\infty}^\beta)\leq[w]_{(A_p)^\al(A_\infty)^\beta}\leq [w]_{A_p}^{\al+\beta}$$
and similarly inequality holds with the exponential constant.  For the exponential class we also have a monotonic behavior in the constants:
$$[w]_{(A_p)^\al(A^{\exp}_\infty)^\beta}\leq[w]_{(A_p)^\al(A_r)^\beta}\leq [w]_{(A_p)^\al(A_s)^\beta} \qquad 1\leq s\leq r< \infty.$$
We also have a monotonic behavior when the power on the $A_p$ part is $0<\al\leq 1$ and the power on the $A^{\exp}_\infty$ part is $1-\al$.  

\begin{observation} \label{obs:incr} Suppose $0<\al\leq \beta\leq 1$ and $w\in A_p$ then
$$[w]_{(A_p)^\al(A^{\exp}_\infty)^{1-\al}}\leq [w]_{(A_p)^\beta(A^{\exp}_\infty)^{1-\beta}}.$$
Moreover, for $\al<\beta$ the two quantities can be arbitrarily different.
\end{observation}
\begin{proof} Indeed, for a fixed cube $Q$
\begin{align*}\lefteqn{A_p(w,Q)^{\al}A^{\exp}_\infty(w,Q)^{1-\al}}\\
&=A_p(w,Q)^{\al}A^{\exp}_\infty(w,Q)^{1-\beta}A^{\exp}_\infty(w,Q)^{\beta-\al}\\
&\leq A_p(w,Q)^{\beta}A^{\exp}_\infty(w,Q)^{1-\beta}.
\end{align*}
To see that $[w]_{(A_p)^\al(A^{\exp}_\infty)^{1-\al}}$ can be arbitrarily smaller than than \\ $[w]_{(A_p)^\beta(A^{\exp}_\infty)^{1-\beta}}$ consider the $A_p$ power weight
$$w_\delta(x)=|x|^{(p-1)(n-\delta)}, \qquad 0<\delta<n.$$
Then for any $\al>0$
$$[w]_{(A_p)^\al(A^{\exp}_\infty)^{1-\al}}\simeq \delta^{-\al(p-1)}$$
which shows that the constants can be arbitrarily different as $\delta\ra 0^+$ if $\al\not=\beta$.
\end{proof}
With this in mind, we may put these estimates into a general frame work.  For the maximal function we have the following question.
\begin{question} \label{ques:max} If $1<p<\infty$ does the estimate 
$$\|M\|_{L^p(w)}\lesssim [\sigma]_{(A_p)^{\frac1p}(A_\infty)^{\frac{1}{p'}}}$$
hold?
\end{question}
The inequality in Question \ref{ques:max} is true if the $A_\infty$ part of the constant is replaced by  exponential $A_\infty$ and Theorems \ref{thm:maxW} shows that it holds up to a logarithmic factor.

For Calder\'on-Zygmund operators we see that there is possibly a whole range of mixed estimates.  We have the following questions.

\begin{question}\label{ques:expal} Suppose $0\leq \al\leq 1$, $1<p<\infty$ and $T$ is a Calder\'on-Zygmund operator.  Do either of the following estimates hold:
$$\|T\|_{L^{p,\infty}(w)}\lesssim [w]_{(A_p)^{\frac{1}{p-\al}}(A^{\exp}_\infty)^{1-\frac{1}{p-\al}}}$$
$$\|T\|_{L^{p}(w)}\lesssim [w]_{(A_p)^{\frac{1}{p-\al}}(A^{\exp}_\infty)^{1-\frac{1}{p-\al}}}+[\sigma]_{(A_{p'})^{\frac{1}{p'-\al}}(A^{\exp}_\infty)^{1-\frac{1}{p'-\al}}}?$$

\end{question}

\begin{question} \label{ques:Wal} Suppose $0\leq \al\leq 1$, $1<p<\infty$ and $T$ is a Calder\'on-Zygmund operator.  Do either of the following estimate hold
$$\|T\|_{L^{p,\infty}(w)}\lesssim [w]_{(A_p)^{\frac{1}{p-\al}}(A_\infty)^{1-\frac{1}{p-\al}}}$$
$$\|T\|_{L^{p}(w)}\lesssim [w]_{(A_p)^{\frac{1}{p-\al}}(A_\infty)^{1-\frac{1}{p-\al}}}+[\sigma]_{(A_{p'})^{\frac{1}{p'-\al}}(A_\infty)^{1-\frac{1}{p'-\al}}}?$$
\end{question}
Theorem \ref{thm:exp1} shows that Question \ref{ques:expal} holds for $\al=1$, while Theorems \ref{thm:exp0} and \ref{thm:W0} show, respectively, that the estimates in Questions \ref{ques:expal} and \ref{ques:Wal} hold for $\al=0$ up to logarithmic factors.  Conjecture \ref{conj:ult1} corresponds to an affirmative answer to Question \ref{ques:Wal} for $\al=0$.  Moreover, by Observation \ref{obs:incr} a positive answer to Question \ref{ques:expal} for $\al=0$ implies the corresponding estimates for $\al>0$.

\subsection{Examples} Let us compare the following bounds of $\|T\|_{L^p(w)}$:
\begin{equation}\label{HL} [w]_{A_p}^{\frac1p}([w]_{A_\infty}^{\frac{1}{p'}}+[\sigma]^{\frac1p}_{A_{\infty}})\end{equation}
\begin{equation}
\label{exp1} [w]_{(A_p)^{\frac{1}{p-1}}(A^{\exp}_\infty)^{1-\frac{1}{p-1}}}   
\end{equation}
\begin{equation}
\label{exp0} \big(\log(e[w]_{A_p})\big)^{\frac{1}{p}}[w]_{(A_p)^{\frac{1}{p}}(A^{\exp}_\infty)^{\frac{1}{p'}}}+\big(\log(e[\sigma]_{A_{p'}})\big)^{\frac{1}{p'}}[\sigma]_{(A_{p'})^{\frac{1}{p'}}(A^{\exp}_\infty)^{\frac{1}{p}}}.
\end{equation}

The constant \eqref{HL} is the bound from \eqref{eqn:LHPmixed}.  While the constants \eqref{exp1} and \eqref{exp0} are the bounds from Theorems \ref{thm:exp1} and \ref{thm:exp0} respectively.  The quantity \eqref{exp1} can be smaller than the quantity \eqref{HL} when $p\not=2$, because it is smaller than the right hand side of \eqref{eqn:mixedpr}, which in turn, was shown to be smaller than \eqref{HL} in \cite{Lern2011}.  Below we will give an example to show that \eqref{exp0} can be arbitrarily smaller than {\it both} \eqref{HL} and \eqref{exp1}.  

Our example is a modification of the one in \cite{Lern2011}.  The example in \cite{Lern2011} is a combination of an $A_1$ power weight and an $A_p$ power weight with the $A_p$ part sufficiently separated from the $A_1$ part.  Our example keeps track of how far apart the $A_p$ part is from the $A_1$ part.  

Consider the case $p>2$ (for the case $p<2$, interchange $p$ and $p'$).  For $0<\delta<1$ and $\frac{1}{p}<\al<\frac12$, define
$$w_\delta(x)=\left\{\begin{array}{ll}|x|^{(p-1)(1-\delta)} & x\in [-1,1] \\ |x-(\delta^{-\al}+1)|^{\delta-1} & x\in [\delta^{-\al},\delta^{-\al}+2] \\ 1 & \text{otherwise}. \end{array}\right.$$
Now $[w_\delta]_{A_p}\gtrsim \delta^{-(p-1)}$ by taking $Q=[0,1]$ and $[w_\delta]_{A_\infty}\gtrsim \delta^{-1}$ by taking $Q=[\delta^{-\al},\delta^{-\al}+1]$.  Therefore,
$$[w_\delta]_{A_p}^{\frac{1}{p}}([w_\delta]_{A_\infty}^{\frac{1}{p'}}+[\sigma_\delta]^{\frac1p}_{A_\infty})\gtrsim\delta^{-\frac{2}{p'}}.$$
On the other hand, since $\al<\frac{1}{2}$, we see that for small $\delta$ the constant 
$$[w_\delta]_{(A_p)^{\frac{1}{p-1}}(A_\infty^{\exp})^{1-\frac{1}{p-1}}}$$
attains it supremum on intervals containing $[0,\delta^{-\al}+1]$ (the smallest interval that contains the singularities of both $w_\delta$ and $\sigma_\delta$).  Thus we see that
$$[w_\delta]_{(A_p)^{\frac{1}{p-1}}(A_\infty^{\exp})^{1-\frac{1}{p-1}}}\simeq \delta^{-2(1-\al)}.$$
The constants $[w_\delta]_{(A_p)^{\frac1p}(A_\infty^{\exp})^{\frac{1}{p'}}}$ and $[\sigma_\delta]_{(A_{p'})^{\frac{1}{p'}}(A_\infty^{\exp})^{\frac1p}}$ either attain their supremum on an interval containing $[0,\delta^{-\al}+1]$ or small intervals containing $0$ or $\delta^{-\al}+1$.  Hence
$$\max\big([w_\delta]_{(A_p)^{\frac1p}(A_\infty^{\exp})^{\frac{1}{p'}}},[\sigma_\delta]_{(A_{p'})^{\frac{1}{p'}}(A_\infty^{\exp})^{\frac1p}}\big)\lesssim \delta^{-\max((1+\frac{1}{p'})(1-\al),1)}.$$
Finally,
$$\log(e[w_\delta]_{A_p})\simeq \log(e[\sigma_\delta]_{A_{p'}})\simeq \log(e\delta^{-1}).$$
In any case, because $\frac1p<\al<\frac12$, we have
$$\max\big((1+\frac{1}{p'})(1-\al),1\big)< 2(1-\al)<\frac{2}{p'}.$$
Letting $\delta \ra 0^+$, we see that \eqref{exp0} can be arbitrarily smaller than both \eqref{exp1} and \eqref{HL}.

\bibliographystyle{plain}
\bibliography{Ref}

\end{document}